\documentclass[11pt,twoside,a4paper]{amsart}
\usepackage{amssymb}
\usepackage{amsmath,amscd}
\usepackage[initials,nobysame,shortalphabetic]{amsrefs}
\usepackage{verbatim}

\def\dbar{\bar\partial}

\def\C{{\mathbb C}}

\newcommand{\Com}[1]{}

\def\be{\begin{equation}}
\def\ee{\end{equation}}

\newtheorem{thm}{Theorem}[section]
\newtheorem{lma}[thm]{Lemma}
\newtheorem{cor}[thm]{Corollary}
\newtheorem{prop}[thm]{Proposition}

\theoremstyle{definition}

\newtheorem{df}{Definition}

\theoremstyle{remark}

\newtheorem{preremark}{Remark}
\newtheorem{preex}{Example}

\newenvironment{remark}{\begin{preremark}}{\end{preremark}}

\numberwithin{equation}{section}

\ifdefined\custompagenr
\fi

\begin{document}

\title[]{The Nakano vanishing theorem and a vanishing theorem of Demailly-Nadel type for holomorphic vector bundles}


\author{Hossein Raufi}

\address{H. Raufi\\Department of Mathematics\\Chalmers University of Technology and the University of Gothenburg\\412 96 G\"OTEBORG\\SWEDEN}

\email{raufi@chalmers.se}




\begin{abstract}
We prove the classical Nakano vanishing theorem with H\"or\-man\-der $L^2$-estimates on a compact K\"ahler manifold using Siu's so called $\partial\dbar$-Boch\-ner-Kodaira method, thereby avoiding the K\"ahler identities completely. We then introduce singular hermitian metrics on holomorphic vector bundles, and proceed to prove a vanishing theorem of Demailly-Nadel type for these in the special case where the base manifold is a Riemann surface.
\end{abstract}

\maketitle

\section{Introduction}
\noindent Let $(X,\omega)$ be a compact K\"ahler manifold with dim$_{\C}X=n$ and let $(E,h)$ be a hermitian, holomorphic vector bundle over $X$ with rank$E=r$. The celebrated Nakano vanishing theorem, which first appeared in \cite{N2}, generalizes the Kodaira vanishing theorem to this setting. Following H\"ormander \cite{H} one can apply $L^2$ methods to prove this vanishing theorem in such a way that one not only obtains existence of solutions to the inhomogenous $\dbar$-equation, but also very accurate and useful $L^2$ estimates for the solutions. More precisely we have the following theorem.

\begin{thm}\label{thm:1.1}
Let $X$ be a compact K\"ahler manifold and let $E$ be a holomorphic vector bundle over $X$. Suppose that $f$ is an $E$-valued, $\dbar$-closed, $(n,p)$-form with $p\geq1$. If $E$ can be equipped with a hermitian metric $h$ which is strictly positively curved in the sense of Nakano with $i\Theta^h\geq_{Nak.}\delta\omega\otimes I$ for some $\delta>0$, then there exists an $E$-valued $(n,p-1)$-form $u$ such that
$$\dbar u=f$$
in the sense of distributions, and furthermore
$$\int_X\|u\|^2dV_\omega\leq\frac{1}{\delta p}\int_X\|f\|^2dV_\omega,$$
provided that the right hand side is finite.
\end{thm}
Here $dV_\omega:=\omega^n/n!$ and the norms of $u$ and $f$ are taken with respect to the metric $h$ and the K\"ahler form $\omega$; see section 2 where these notions are reviewed.

In the usual proofs of Theorem \ref{thm:1.1} the so called K\"ahler identities are a key ingredient. However in \cite{S1} Siu introduced a method which he calls the $\partial\dbar$-Bochner-Kodaira method and which manages to avoid the (rather non-intuitive) K\"ahler identities completely.

In Riemannian geometry the basic idea behind the Bochner method is (very vaguely) to calculate the Laplacian of the norm of forms. Then one can draw conclusions about the geometry by carefully analyzing the resulting expression and putting restrictions on the curvature of the metric. The straightforward adaptation of this method in our complex setting would then be to calculate and analyze
\be\label{eq:lapl}
\Delta\|\alpha\|^2
\ee
where $\alpha$ is an $E$-valued, $(n,p)$-form. However, it turns out that this approach will not work out well and so the historical approach to the vanishing theorem has been through the K\"ahler identities.

What Siu demonstrates in \cite{S1} (among other things) is that if the metric is dually, negatively curved in the sense of Nakano, an approach that is very similar to the classical Bochner method can be applied. The main idea is to let the $E$-valued $(0,q)$-form $\alpha$ remain form-valued, replace $\Delta$ by $i\partial\dbar$ and calculate
$$i\partial\dbar c_q\langle\alpha,\alpha\rangle\wedge\omega^{n-q-1}/(n-q-1)!$$
instead of (\ref{eq:lapl}). Here $c_q:=i^{q^2}$ is chosen to make the right hand side positive and $\langle,\rangle$ is the bilinear form on $E$-valued forms induced by $h$ (see section 2). In \cite{S1} section 3 it is shown that the resulting expression is equivalent to the classical Bochner-Kodaira-Nakano identity, and this is then used in section 4 to prove (a slightly more general version of) the Nakano vanishing theorem, but without any $L^2$ estimates. 

In \cite{B} Berndtsson shows that in the line bundle case, this method can be applied directly, without resorting to dual bundles, and he also derives the H\"ormander $L^2$ estimates. Here the situation is slightly more involved. Let $(L,\phi)$ be a positively curved line bundle over $X$ and let $\alpha$ be an $(n,p)$-form with values in $L$. It turns out that the appropriate counterpart of (\ref{eq:lapl}) in this case is
$$i\partial\dbar c_{n-p}\gamma_\alpha\wedge\gamma_\alpha\wedge\omega^{p-1}e^{-\phi}/(p-1)!$$ 
where $\gamma_\alpha$, (up to a constant), is the Hodge-$\ast$ of $\alpha$, i.e. an $L$-valued $(n-p,0)$-form such that
$$\alpha=\gamma_\alpha\wedge\omega^p/p!.$$

The first aim of this paper is to show that this latter approach works almost without change for forms with values in a vector bundle, thereby proving Theorem \ref{thm:1.1}. Our presentation will follow that of \cite{B} rather closely but we still believe that this is of interest since it simplifies the otherwise, in our view, rather messy proofs and gives a more natural approach to the pretty mysterious notion of being curved in the sense of Nakano. We hope that this illustrates the usefulness of the $\partial\dbar$-Bochner-Kodaira method, which we believe deserves to be more well-known. Another application of this method is given in \cite{R2} where we use it to prove a Ohsawa-Takegoshi type of extension theorem for sections of a vector bundle.

Once the compact K\"ahler case is in place, one can extend Theorem \ref{thm:1.1} to a large class of non-compact complex manifolds, namely those carrying some complete K\"ahler metric. In particular, Theorem \ref{thm:1.1} holds on Stein manifolds (see e.g. \cite{D1} Chapter VIII, Theorem 4.5).

For line bundles an important extension of Theorem \ref{thm:1.1} is the Demailly-Nadel vanishing theorem which states that Theorem \ref{thm:1.1} holds on any complete K\"ahler manifold even if we only assume that the metric is locally integrable. The curvature is then required to be positive or negative as a current, (\cite{D2},\cite{N1}).

In a previous article \cite{R1} we have, following Berndtsson-P{\u{a}}un \cite{BP}, introduced and studied the notion of singular hermitian metrics on holomorphic vector bundles (see also de Cataldo \cite{C}). The definition that we have adopted is the following.

\begin{df}\label{df:sing}
Let $E\to X$ be a holomorphic vector bundle over a complex manifold $X$. A singular hermitian metric $h$ on $E$ is a measurable map from the base space $X$ to the space of non-negative hermitian forms on the fibers.
\end{df}

On a line bundle $h$ is (locally) just a scalar-valued function so the curvature $\Theta=\dbar(h^{-1}\partial h)=-\partial\dbar\log h$ is well-defined as a current as long as $\log h\in L^1_{loc}(X)$. As $h$ is matrix-valued on a vector bundle, it is not clear what $\dbar(h^{-1}\partial h)$ means if $h$ is singular. However for smooth metrics there exists equivalent characterisations of curvature in the sense of Griffiths and Nakano, and these are the ones that we utilize in the singular setting.
 
\begin{df}\label{df:2.2}
Let $E\to X$ be a holomorphic vector bundle over a complex manifold $X$ and let $h$ be a singular hermitian metric. We say that $h$ is negatively curved in the sense of Griffiths if $\|u\|^2_h$ is plurisubharmonic for any holomorphic section $u$. Furthermore we say that $h$ is positively curved in the sense of Griffiths if the dual metric is negatively curved.
\end{df}

This is a very natural definition as these conditions both are well-known equivalent properties for smooth metrics; see section 2 where these facts are reviewed. In \cite{R1} our goal is to try to make sense of the curvature tensor associated to a singular hermitian metric $h$, that is positively or negatively curved in the sense of Griffiths as in Definition \ref{df:2.2}. The main result is that if we assume that $F:=\{\det h=0\}$ is a closed set, and furthermore assume that there exists an exhaustion of open sets $\{\Omega_j\}$ of $F^c$ such that $\det h>1/j$ on $\Omega_j$, then $\Theta^h:=\dbar(h^{-1}\partial h)$ can be defined as a current on $F^c$ (\cite{R1} Theorem 1.1). We also give an example which shows that it is not possible to define $\Theta^h$ in this way on the singular locus of $h$.

Once the existence of $\Theta^h$ as a current is in place, we can proceed to define what it means for $\Theta^h$ to be strictly positively or negatively curved in the sense of Griffiths. The definition that we will adopt is the following.

\begin{df}\label{df:2.3}
Let $h$ be a singular hermitian metric on a holomorphic vector bundle $E$. Assume that $F:=\{\det h=0\}$ is a closed set, and that there exists an exhaustion of open sets $\{\Omega_j\}$ of $F^c$ such that $\det h>1/j$ on $\Omega_j$. We say that $h$ is strictly negatively curved in the sense of Griffiths if:\\
(i) $h$ is negatively curved in the sense of Definition \ref{df:2.2}. In particular, $\Theta^h$ exists as a current on $F^c$.\\
(ii) there exists some $\delta>0$ such that on $F^c$
\be\label{eq:GrStr}
\sum_{j,k=1}^n\big(\Theta_{jk}^hs,s\big)_h\xi_j\bar{\xi}_k\leq-\delta\|s\|^2_h|\xi|^2
\ee
for any section $s$ and any vector $\xi\in\C^n$.

We say that $h$ is strictly positively curved in the sense of Griffiths, if the corresponding dual metric is strictly negatively curved.
\end{df}

Here $\{\Theta_{jk}\}$ are the matrix-valued distributions defined through
$$\Theta^h=\sum_{j,k=1}^n\Theta_{jk}^hdz_j\wedge d\bar{z}_k$$
and so the expression in (\ref{eq:GrStr}) should be interpreted in the sense of distributions.

The usefulness of Definition \ref{df:2.3} comes from the fact that it lends itself well to regularisations, which is the key ingredient needed in order to adapt the line bundle proofs of the Demailly-Nadel vanishing theorem to the vector bundle setting.

However, in the latter setting, strict positivity in the sense of Nakano is needed. We could proceed to define what this means on $F^c$ just as in the Griffiths case, but for the definition to be useful, we need some way of approximating the singular metric with a sequence of strictly Nakano positive metrics.

The key points in the regularisation arguments for singular hermitian metrics that are strictly positively or negatively curved in the sense of Griffiths, are the alternative characterisation of Griffiths negativity in terms of the plurisubharmonicity of $\|u\|^2_h$ for any holomorphic section $u$, and the fact that the dual of a Griffiths negative bundle is Griffiths positive. An alternative characterisation of Nakano negativity also exists, and in \cite{R1} we in fact prove a regularisation result in the Nakano negative setting (\cite{R1} Proposition 1.3). However, unlike curvature in the sense of Griffiths, the dual of a Nakano negative bundle, in general, is not Nakano positive. Due to this rather surprising fact we can not use this regularisation result in the positive setting. Hence for singular metrics a whole new approach to Nakano positivity is probably needed. Unfortunately we have so far failed to come up with an appropriate definition where regularisation is possible.

However for vector bundles over Riemann surfaces, i.e. when dim$_\C X=1$, the notions of Griffiths and Nakano curvature coincide. Hence we can prove the following Demailly-Nadel type of vanishing theorem in section 4.

\begin{thm}\label{thm:1.3}
Let $X$ be a Riemann surface and let $E$ be a holomorphic vector bundle over $X$. Let furthermore $h$ be a continuous singular hermitian metric on $X$. Suppose that $f$ is an $E$-valued, $(1,1)$-form. If $h$ is positively curved as in Definition \ref{df:2.3}, normalized so that $\delta=1$, there exists an $E$-valued $(1,0)$-form $u$ such that
$$\dbar u=f$$
in the sense of distributions, and
$$\int_X\|u\|_{h,\omega}^2dV_\omega\leq\int_X\|f\|_{h,\omega}^2dV_\omega,$$
provided that the right hand side is finite.
\end{thm}

\section*{Acknowledgments}
\noindent It is a pleasure to thank my advisor Bo Berndtsson for inspiring and fruitful discussions.

\section{The setting} 
\noindent Let $(X,\omega)$ be a compact K\"ahler manifold and let $(E,h)$ be a hermitian, holomorphic vector bundle over $X$. Then we have a well-defined bilinear form which we denote by $\langle,\rangle$, for forms on $X$ with values in $E$ by letting $\langle\alpha\otimes s,\beta\otimes t\rangle:=\alpha\wedge\bar{\beta}\ (s,t)_h$ for forms $\alpha,\beta$ and sections $s,t$, and then extend to arbitrary forms with values in $E$ by linearity. We denote the Chern connection associated to this bilinear form by $D=D'+\dbar$ and the curvature by $\Theta=D^2=D'\dbar+\dbar D'$.

Now let $\alpha$ be an $E$-valued form of bidegree $(p,0)$. We define the norm of $\alpha$ with respect to the metrics $h$ and $\omega$ through
\be\label{eq:norm}
\|\alpha\|^2dV_\omega=c_p\langle\alpha,\alpha\rangle\wedge\omega_{n-p}
\ee
where $dV_\omega:=\omega^n/n!$, $\omega_{n-p}:=\omega^{n-p}/(n-p)!$ and $c_p:=i^{p^2}$ an unimodular constant chosen so that the right hand side is positive. One can show that if $\{dz_j\}$ are orthonormal coordinates at a point and
$$\alpha=\sum_{|I|=p}\alpha_Idz_I$$
where $\{\alpha_I\}$ are sections of $E$, then
$$\|\alpha\|^2=\sum_{|I|=p}\|\alpha_I\|^2_h.$$
We also use (\ref{eq:norm}) to define the norm of $E$-valued forms of bidegree $(0,q)$. In particular then $\|\alpha\|=\|\bar{\alpha}\|$.

Using this definition one can now proceed to show that it is possible to define the norm of an $E$-valued form $\eta$ of arbitrary bidegree in such a way that if
$$\eta=\sum\eta_{IJ}dz_I\wedge d\bar{z}_J$$
in terms of an orthonormal basis at a point, then
$$\|\eta\|^2=\sum\|\eta_{IJ}\|^2_h.$$
This norm can then be polarized yielding an inner product for $E$-valued forms of arbitrary bidegree. Hence if $\xi$ is another form with values in $E$, which is of the same bidegree as $\eta$, then
$$(\eta,\xi)=\sum(\eta_{IJ},\xi_{IJ})_h$$
if we express $\eta$ and $\xi$ in terms of an orthonormal basis as above.

With respect to this inner product we can then define the formal adjoint of the $\dbar$ operator with respect to the metric $h$ through
\be\label{eq:adjoint}
\int_X(\dbar\eta,\xi)dV_\omega=\int_X(\eta,\dbar^*_h\xi)dV_\omega
\ee
for all $E$-valued forms $\eta$ and $\xi$ of appropriate bidegrees.

In presenting the inner product on $E$-valued forms of arbitrary bidegree we have been rather sketchy since we will only be interested in forms of bidegree $(n,p)$ and for these there exists a useful special formula. Namely given any $(n,p)$-form $\alpha$ it follows from a computation in orthonormal coordinates that there exists an $(n-p,0)$-form $\gamma_\alpha$ such that
$$\alpha=\gamma_\alpha\wedge\omega_p.$$
If $\alpha$ is given in orthonormal coordinates by
$$\alpha=\sum_{|J|=p}\alpha_Jdz\wedge d\bar{z}_J$$
then $\gamma_\alpha$ will be given by
$$\gamma_\alpha=\sum_{|J|=p}\varepsilon_J\alpha_Jdz_{J^c}$$
where $\varepsilon_J$ are unimodular constants. In terms of Riemannian geometry one usually calls $\gamma_\alpha$ the Hodge-$\ast$ of $\alpha$.

It is immediate that the existence of $\gamma_\alpha$ is not affected by requiring $\alpha$ to be $E$-valued and furthermore it is clear that in this case
$$\|\alpha\|^2=\|\gamma_\alpha\|^2.$$
Together with (\ref{eq:norm}) this in turn implies that
$$c_{n-p}\langle\alpha,\gamma_\alpha\rangle=c_{n-p}\langle\gamma_\alpha,\gamma_\alpha\rangle\wedge\omega_p=\|\gamma_\alpha\|^2dV_\omega=\|\alpha\|^2dV_\omega$$
and polarizing this we arrive at
\be\label{eq:product}
(\alpha,\beta)dV_\omega=c_{n-p}\langle\alpha,\gamma_\beta\rangle
\ee
for any other $E$-valued, $(n,p)$-form $\beta$.

Using this last formula we can deduce a very useful relation between the formal adjoint of $\dbar$ and the $(1,0)$-part of the Chern connection. If we let $\alpha$ be an $E$-valued $(n,p-1)$-form but keep $\beta$ as before we have that on the one hand
$$\int_X(\dbar\alpha,\beta)dV_\omega = c_{n-p}\int_X\langle\dbar\alpha,\gamma_\beta\rangle=c_{n-p}(-1)^{n-p}\int_X\langle\alpha,D'\gamma_\beta\rangle$$
by Stokes' theorem and on the other hand
$$\int_X(\alpha,\dbar_h^*\beta)dV_\omega=c_{n-p+1}\int_X\langle\alpha,\gamma_{\dbar_h^*\beta}\rangle=ic_{n-p}(-1)^{n-p}\int_X\langle\alpha,\gamma_{\dbar_h^*\beta}\rangle.$$
Hence we see that
$$D'\gamma_\beta=-i\gamma_{\dbar_h^*\beta}$$
so that in particular
\be\label{eq:adj}
\|D'\gamma_\beta\|^2=\|\gamma_{\dbar_h^*\beta}\|^2=\|\dbar_h^*\beta\|^2.
\ee

Given a hermitian metric $h$ on a holomorphic vector bundle $E$, the curvature $\Theta$ associated to $h$ is locally a matrix of $(1,1)-$forms which we write as
\be\label{eq:curv2}
\Theta=\sum_{j,k=1}^n\Theta_{jk}dz_j\wedge d\bar{z}_k
\ee
where $\Theta_{jk}$ are $r\times r$ matrix-values functions on $X$, $r$ being the rank of $E$. We then say that $(E,h)$ is strictly positively curved in the sense of Griffiths if for some $\delta>0$
$$\big(i\Theta(\xi,\xi)s,s\big)_h\geq\delta\|s\|^2_h|\xi|^2$$
for any section $s$ of $E$, and any smooth vector field $\xi$. Using (\ref{eq:curv2}) we see that this is equivalent to
$$\sum_{j,k=1}^n\big(\Theta_{jk}s,s\big)_h\xi_j\bar\xi_k\geq\delta\|s\|^2_h|\xi|^2$$
for any vector $\xi\in\C^n$. We will use the shorthand notation $i\Theta\geq_{Gr.}\delta\omega\otimes I$ to express this.

We say that $E$ is strictly positively curved in the sense of Nakano if for some $\delta>0$
\be\label{eq:nak2}
\sum_{j,k=1}^n\big(\Theta_{jk}s_j,s_k\big)_h\geq\delta\sum_{j=1}^n\|s_j\|^2_h
\ee
for any $n$-tuple $(s_1,\ldots,s_n)$ of sections of $E$. Griffiths and Nakano semipositivity, seminegativity and strict negativity are defined similarly, and in this setting we will use the shorthand notation $i\Theta\geq_{Nak.}\delta\omega\otimes I$.

Choosing $s_j=\xi_js$ in (\ref{eq:nak2}) it is immediate that being positively or negatively curved in the sense of Nakano implies being positively or negatively curved in the sense of Griffiths. The converse however, does not hold in general. Of these two main positivity concepts Griffiths positivity has the nicest functorial properties in that if $E$ is positively curved in the sense of Griffiths, then the dual bundle $E^*$ has negative Griffiths curvature. This property will be very useful for us as it allows us to study metrics with positive and negative Griffiths curvature interchangeably. Unfortunately this correspondence between $E$ and $E^*$ does not hold in the Nakano case. The reason for studying Nakano positivity is that it is intimately related with the solvability of the inhomogeneous $\dbar$-equation using $L^2$ methods.

Let $u$ be an arbitrary holomorphic section of $E$. Then a short computation yields
$$i\partial\dbar\|u\|_{h}^2=-\langle i\Theta u,u\rangle_{h}+i\langle D^\prime u,D^\prime u\rangle_{h}\geq-\langle i\Theta u,u\rangle_{h}$$
Hence we see that if the curvature is negative in the sense of Griffiths, then $\|u\|^2_{h}$ is plurisubharmonic. On the other hand we can always find a holomorphic section $u$ such that $D^\prime u=0$ at a point. Thus $h$ is negatively curved in the sense of Griffiths if and only if $\|u\|^2_{h}$ is plurisubharmonic, for any holomorphic section $u$. This is the motivation behind Definition \ref{df:2.2}.

Now if $\gamma$ is an $E$-valued $(n-1,0)$-form we can locally write it as
$$\gamma=\sum_{j=1}^n\gamma^j\widehat{dz_j}$$
where $\widehat{dz_j}$ denotes the wedge product of all $dz_k$ except $dz_j$ ordered so that $dz_j\wedge\widehat{dz_j}=dz_1\wedge\ldots\wedge dz_n$. One can then verify that
\be\label{eq:Nakan}
ic_{n-1}\langle\Theta\wedge\gamma,\gamma\rangle=\sum_{j,k=1}^n\big(\Theta_{jk}\gamma^j,\gamma^k\big)dV_\omega.
\ee
Hence if $\Theta$ is strictly positively curved in the sense of Nakano then
$$ic_{n-1}\langle\Theta\wedge\gamma,\gamma\rangle\geq\delta\|\gamma\|^2dV_\omega,$$
for all $E$-valued $(n-1,0)$-forms $\gamma$.

More generally one can show that if $\Theta$ is strictly positively curved in the sense of Nakano, then for any $E$-valued $(n-p,0)$-form $\gamma$
\be\label{eq:BasicIneq}
ic_{n-p}\langle\Theta\wedge\gamma,\gamma\rangle\wedge\omega_{p-1}\geq\delta p\|\gamma\|^2dV_\omega.
\ee
To see this we first note that if $i\Theta\geq_{Nak.}0$ implies that
\be\label{eq:BI_light}
ic_{n-p}\langle\Theta\wedge\gamma,\gamma\rangle\wedge\omega_{p-1}\geq0,
\ee
we are done, since $i\Theta\geq_{Nak.}\delta\omega\otimes I$ in that case yields
{\setlength\arraycolsep{2pt}
\begin{eqnarray}
&&ic_{n-p}\langle\Theta\wedge\gamma,\gamma\rangle\wedge\omega_{p-1}\geq\delta c_{n-p}\langle\omega\wedge\gamma,\gamma\rangle\wedge\omega_{p-1}=\nonumber\\
&&\quad=\delta c_{n-p}\langle\gamma,\gamma\rangle\wedge\frac{\omega^p}{(p-1)!}=\delta pc_{n-p}\langle\gamma,\gamma\rangle\wedge\omega_p=\delta p\|\gamma\|^2dV_\omega.\nonumber
\end{eqnarray}}
\!\!Thus (\ref{eq:BasicIneq}) is equivalent to proving that $i\Theta\geq_{Nak.}0$ implies (\ref{eq:BI_light}).

In terms of an orthonormal basis $\{dz_j\}$, we have that
$$\omega_{p-1}=\sum_{|J|=p-1}dV_J$$
where
$$dV_J=\bigwedge_{j\in J}idz_j\wedge d\bar{z}_j=c_{p-1}dz_J\wedge d\bar{z}_J.$$
The last equality follows from the fact that $dV_J$ is a positive $(p-1,p-1)$-form, and hence the constant $c_{p-1}$ is needed to turn the right hand side into a positive form as well. Inserting this gives us the sum
$$ic_{n-p}\langle\Theta\wedge\gamma,\gamma\rangle\wedge\omega_{p-1}=ic_{n-p}c_{p-1}(-1)^{(n-p)(p-1)}\!\!\!\!\sum_{|J|=p-1}\!\!\!\langle\Theta\wedge\gamma\wedge dz_J,\gamma\wedge dz_J\rangle.$$
One can now verify that $c_{n-p}c_{p-1}(-1)^{(n-p)(p-1)}=c_{n-1}$ and so we get that each term in the sum is of the form
$$ic_{n-1}\langle\Theta\wedge\tilde{\gamma},\tilde{\gamma}\rangle$$
where $\tilde{\gamma}$ is an $E$-valued, $(n-1,0)$-form. However, from the calculations above we see that if $i\Theta\geq_{Nak.}0$, each such term is non-negative. This shows that (\ref{eq:BI_light}), and hence (\ref{eq:BasicIneq}), holds.

\section{The compact K\"ahler case}
\noindent In this section we will prove part (i) of Theorem \ref{thm:1.1}. The general functional analytic result that we will utilize is the following.
\begin{thm}\label{thm:3.1}
Let $T$ be a linear operator between two Hilbert spaces $H_1$ and $H_2$, and let $F$ be a closed subspace of $H_2$ containing the range of $T$. Assume furthermore that $T$ is closed and densely defined. Then $T$ is surjective onto $F$ if and only if there is a constant $c>0$ such that
\be\label{eq:Hilbineq}
c\|y\|_{H_2}^2\leq\|T^*y\|^2_{H_1}
\ee
for all $y\in Dom(T^*)\cap F$.

In that case for any $y\in F$ there is an $x\in H_1$ such that $Tx=y$ and
$$\|x\|_{H_1}^2\leq\frac{1}{c}\|y\|_{H_2}^2.$$
\end{thm}
\noindent For a proof we refer to chapter 2 in \cite{B}.

In our setting of course $T=\dbar$ and $F=Ker(\dbar)$. Since we are only interested in the case when $E$ carries a metric which is positively curved in the sense of Nakano our Hilbert spaces will be $H_1=L^2_{(n,p-1)}(X,E,h)$ and $H_2=L^2_{(n,p)}(X,E,h)$ for $p\geq1$.

What we want to show to begin with is that for any smooth, $E$-valued, $(n,p)$-form $\alpha$ with $p\geq1$, there exists a constant $c>0$ such that
$$c\int_X\|\alpha\|^2dV_\omega\leq\int_X(\|\dbar_h^*\alpha\|^2+\|\dbar\alpha\|^2)dV_\omega.$$
Traditionally this is achieved through the K\"ahler identities but as mentioned in the introduction we are going to use Siu's $\partial\dbar$-Bochner-Kodaira method. Hence given $\alpha$ we will compute and analyze
$$i\partial\dbar c_{n-p}\langle\gamma,\gamma\rangle\wedge\omega_{p-1}$$
where $\gamma$ is the Hodge-$\ast$ of $\alpha$ (as there is no risk for confusion, in what follows we will write $\gamma$ instead of $\gamma_\alpha$). This is the content of the following proposition.

\begin{prop}
Let $(X,\omega)$ be a K\"ahler manifold and let $(E,h)$ be a hermitian, holomorphic vector bundle over $X$. Then for any smooth, $E$-valued $(n,p)$-form $\alpha$ on $X$ with $\alpha=\gamma\wedge\omega_p$
{\setlength\arraycolsep{2pt}
\begin{eqnarray}\label{eq:basic}
ic_{n-p}\partial\dbar\langle\gamma,\gamma\rangle\!\wedge\omega_{p-1}\!&=&\! ic_{n-p}\Big(\!\langle\Theta\wedge\gamma,\gamma\rangle\!-\!\langle\dbar D'\gamma,\gamma\rangle\!+\!\langle\gamma,\dbar D'\gamma\rangle\!\Big)\!\wedge\omega_{p-1}  \nonumber\\
& & {}+\Big(\|\dbar_h^*\alpha\|^2+\|\dbar\gamma\|^2-\|\dbar\alpha\|^2\Big)dV_\omega.
\end{eqnarray}}

In particular, if $\alpha$ has compact support
\be\label{eq:IntBasic}
ic_{n-p}\int_X\langle\Theta\gamma,\gamma\rangle\wedge\omega_{p-1}+\int_X\|\dbar\gamma\|^2dV_\omega=\int_X(\|\dbar_h^*\alpha\|^2+\|\dbar\alpha\|^2)dV_\omega.
\ee
\end{prop}

\begin{proof}
A short computation using the compatibility of the Chern connection gives
{\setlength\arraycolsep{2pt}
\begin{eqnarray}
\partial\dbar\langle\gamma,\gamma\rangle&=&\langle D'\dbar\gamma,\gamma\rangle+\langle\gamma,\dbar D'\gamma\rangle+\nonumber\\
& & {}+(-1)^{n-p}\langle D'\gamma,D'\gamma\rangle+(-1)^{n-p+1}\langle\dbar\gamma,\dbar\gamma\rangle\nonumber
\end{eqnarray}}
Using that $\Theta=D'\dbar+\dbar D'$ in the first term then yields the first three terms in (\ref{eq:basic}). Now if we write $D'\gamma$ in orthonormal coordinates as
$$D'\gamma=\sum_{|J|=n-p+1}(D'\gamma)_Jdz_J$$
we see that
{\setlength\arraycolsep{2pt}
\begin{eqnarray}
ic_{n-p}(-1)^{n-p}\langle D'\gamma,D'\gamma\rangle\wedge\omega_{p-1} & \!=\! & c_{n-p+1}\!\!\!\!\!\sum_{|J|=n-p+1}\!\!\!\!\!\|(D'\gamma)_J\|^2_hdz_J\wedge d\bar{z}_J\wedge\omega_{q-1}=\nonumber\\
& = & \|D'\gamma\|^2dV_\omega=\|\dbar_h^*\alpha\|^2dV_\omega\nonumber
\end{eqnarray}}
where we have used (\ref{eq:adj}) and the fact that $ic_{n-p}(-1)^{n-p}=c_{n-p+1}$.

Only the last term remains but this one is much more tricky than the rest. We will only treat the most important case $p=1$. The general case is similar, but more tedious.

We will show that for any $E$-valued $(n-1,1)$-form $\xi$
\be\label{eq:equa}
ic_{n-1}(-1)^n\langle\xi,\xi\rangle=(\|\xi\|^2-\|\xi\wedge\omega\|^2)dV_\omega.
\ee
This will then imply that
$$ic_{n-1}(-1)^n\langle\dbar\gamma,\dbar\gamma\rangle=(\|\dbar\gamma\|^2-\|\dbar\gamma\wedge\omega\|^2)dV_\omega=(\|\dbar\gamma\|^2-\|\dbar\alpha\|^2)dV_\omega$$
where the last step uses that $\omega$ is a K\"ahler form. (For $p=1$ this is in fact the only place where the K\"ahler assumption is used.) 

Now in an orthonormal basis $\xi$ is of the form
$$\xi=\sum_{j,k}\xi_{jk}\widehat{dz_j}\wedge d\bar{z}_k$$
where $\widehat{dz_j}$ denotes the product of all differentials $dz_l$ except $dz_j$, ordered so that $dz_j\wedge\widehat{dz_j}=dz$. Thus
$$\langle\xi,\xi\rangle=-\sum_{j,k}(\xi_{jk},\xi_{kj})_hdz\wedge d\bar{z}$$
so that
$$ic_{n-1}(-1)^n\langle\xi,\xi\rangle=\sum_{j,k}(\xi_{jk},\xi_{kj})_hdV_\omega$$
where we have once again used that $ic_{n-1}(-1)^{n-1}=c_n$. On the other hand, in orthonormal coordinates $\omega$ is of the form
$$\omega=i\sum_{j}dz_j\wedge d\bar{z}_j$$
and so
$$\xi\wedge\omega=i(-1)^{n+1}\sum_{j<k}(\xi_{jk}-\xi_{kj})dz\wedge dz_j\wedge dz_k.$$
Also a short computation gives that
$$\|\xi\wedge\omega\|^2=\sum_{j<k}\|\xi_{jk}-\xi_{kj}\|^2_h=\sum_{j,k}\Big(\|\xi_{jk}\|^2-(\xi_{jk},\xi_{kj})\Big).$$
Altogether this proves (\ref{eq:equa}) thereby finishing the proof of (\ref{eq:basic}).

To get to (\ref{eq:IntBasic}) from (\ref{eq:basic}) we just need to compute the integral of the second and third terms. Now this is rather straighforward using the compatibility of the Chern connection, Stokes' theorem and the fact that $ic_{n-p}(-1)^{n-p}=c_{n-p+1}$, namely
{\setlength\arraycolsep{2pt}
\begin{eqnarray}
&&-ic_{n-p}\int_X\langle\dbar D'\gamma,\gamma\rangle\wedge\omega_{p-1} = -ic_{n-p}\int_X\dbar\langle D'\gamma,\gamma\rangle\wedge\omega_{p-1}-\nonumber\\
& & {}\qquad-ic_{n-p}(-1)^{n-p}\int_X\langle D'\gamma,D'\gamma\rangle\wedge\omega_{p-1}=\nonumber\\
& & {}\qquad=-\int_X\|D'\gamma\|^2dV_\omega=-\int_X\|\dbar_h^*\alpha\|^2dV_\omega\nonumber
\end{eqnarray}}
and in the same way
{\setlength\arraycolsep{2pt}
\begin{eqnarray}
&&ic_{n-p}\int_X\langle\gamma,\dbar D'\gamma\rangle\wedge\omega_{p-1} = ic_{n-p}(-1)^{n-p}\int_X\partial\langle \gamma,D'\gamma\rangle\wedge\omega_{p-1}-\nonumber\\
& & {}\qquad-ic_{n-p}(-1)^{n-p}\int_X\langle D'\gamma,D'\gamma\rangle\wedge\omega_{p-1}=-\int_X\|\dbar_h^*\alpha\|^2dV_\omega.\nonumber
\end{eqnarray}}
\end{proof}

\begin{cor}\label{cor:main}
Let $(X,\omega)$ be a compact K\"ahler manifold and let $(E,h)$ be a hermitian, holomorphic vector bundle over $X$, where $h$ is strictly positively curved in the sense of Nakano, with $i\Theta^h\geq_{Nak.}\delta\omega\otimes I$ for some $\delta>0$. Then for any smooth, $E$-valued, $(n,p)$-form $\alpha$ with $p\geq1$
$$p\delta\int_X\|\alpha\|^2dV_\omega\leq\int_X(\|\dbar^*_h\alpha\|^2+\|\dbar\alpha\|^2)dV_\omega.$$
\end{cor}

\begin{proof}
Recall that being positively curved in the sense of Nakano means that there exists some $\delta>0$ such that
$$\sum_{j,k}(\Theta_{jk}s_j,s_k)_h\geq\delta\sum_j\|s_j\|^2_h$$
for any $n-$tuple $(s_1,\ldots,s_n)$ of sections of $E$. From (\ref{eq:BasicIneq}) we also know that this in particular implies that
$$ic_{n-p}\langle\Theta\wedge\gamma,\gamma\rangle\wedge\omega_{p-1}\geq\delta p\|\gamma\|^2dV_\omega$$
for any $E$-valued, $(n-p,0)$-form $\gamma$. Applying this to (\ref{eq:IntBasic}) and using the fact that $\|\gamma_\alpha\|^2=\|\alpha\|^2$ then yields the desired result.
\end{proof}

Now as already mentioned our strategy is to use Theorem \ref{thm:3.1}, and so we let $T$ denote the closed and densely defined operator $\dbar$ from $H_1=L^2_{(n,p-1)}(X,E,h)$ to $H_2=L^2_{(n,p)}(X,E,h)$ for $p\geq1$. What we need to show is that for some constant $c>0$
$$c\int_X\|\alpha\|^2dV_\omega\leq\int_X\|T^*\alpha\|^2dV_\omega$$
for all $\alpha\in Dom(T^*)\cap Ker(\dbar)$ where $T^*$ is the Hilbert space adjoint of $T$. In general $T^*$ is different from the formal adjoint $\dbar^*_h$ that we have studied so far, in that $T^*$ has the specific domain
$$Dom(T^*)=\Big\{\alpha\in H_2\ ; \Big|\!\!\int_X(\dbar u,\alpha)dV_\omega\Big|^2\leq C\int_X\|u\|^2dV_\omega\quad\forall u\in Dom(T)\Big\}$$
where $u\in Dom(T)$ means that $\dbar u$, in the sense of distributions, lies in $H_2$.

Note that any smooth, $E$-valued, $(n,p)$-form $\alpha$ lies in $Dom(T^*)$ and furthermore
$$T^*\alpha=\dbar^*_h\alpha.$$
This follows immediately from the definition of $Dom(T^*)$ since if $u\in H_1$ is such that $\dbar u\in H_2$, then for any smooth $\alpha\in H_2$
$$\int_X(\dbar u,\alpha)dV_\omega=\int_X(u,\dbar^*_h\alpha)dV_\omega$$
as this is precisely the definition of $\dbar u$ in the sense of distributions.

The remaining key ingredient is now given by the following important approximation lemma.

\begin{lma}\label{lma:appr}
If $\alpha\in Dom(T^*)$ is such that $\dbar\alpha\in L^2_{(n,p+1)}(X,E,h)$, then there exists a sequence $\{\alpha_k\}$ of smooth, $E$-valued, $(n,p)$-forms such that
$$\int_X\|\alpha_k-\alpha\|^2dV_\omega,$$
$$\int_X\|\dbar\alpha_k-\dbar\alpha\|^2dV_\omega$$
and
$$\int_X\|\dbar^*_h\alpha_k-T^*\alpha\|^2dV_\omega$$
all tend to zero.
\end{lma}

\begin{proof}
We want to approximate $\alpha$ as usual by taking the convolution with a sequence of functions
$$\chi_k(z):=k^{2n}\chi(kz)$$
where $\chi$ is a smooth function with compact support and integral equal to 1. However, as $\alpha$ is $E$-valued, we first need to show that we may assume that $\alpha$ is supported in a coordinate neighborhood that trivializes $E$. This can be achieved through an appropriate partition of unity and so what we need to show is that if $\alpha\in Dom(T^*)$, and if $\varphi$ is a smooth, real valued function on $X$, then $\varphi\alpha$ also lies in $Dom(T^*)$.

By definition then, we need to prove that for any $u\in Dom(T)$
$$\Big|\int_X(\dbar u,\varphi\alpha)dV_\omega\Big|^2\leq C\int_X\|u\|^2dV_\omega.$$
But this follows from the fact that
{\setlength\arraycolsep{2pt}
\begin{eqnarray}
 \int_X(\dbar u,\varphi\alpha)dV_\omega & = & \int_X(\varphi\dbar u,\alpha)dV_\omega=\nonumber\\
& = & \int_X\big(\dbar(\varphi u),\alpha\big)dV_\omega-\int_X(\dbar\varphi\wedge u,\alpha)dV_\omega.\nonumber
\end{eqnarray}}
\!\!The first term is fine as $\alpha\in Dom(T^*)$ and the second term is fine as well since $\varphi$ is smooth, and hence it is dominated by the inner product of $u$ and $\alpha$. 

Hence we may assume that $E$ is trivial so that the section part of $\alpha$ is a vector of functions. Thus if $\alpha$ locally is of the form
$$\alpha=\sum_{|J|=p}\left[ \begin{array}{c} \alpha^1_J\\\vdots\\\alpha^r_J\end{array} \right]dz\wedge d\bar{z}_J$$
where $dz=dz_1\wedge\ldots\wedge dz_n$, we let
$$\alpha_k:=\sum_{|J|=p}\left[ \begin{array}{c} \alpha^1_J\ast\chi_k\\\vdots\\\alpha^r_J\ast\chi_k\end{array} \right]dz\wedge d\bar{z}_J.$$

Now it is a standard property of such convolutions that for any $f\in L^p$
$$f\ast\chi_k\to f\quad\textrm{in }\ L^p.$$
As our metric $h$ is smooth on $X$, which is compact, $h$ is bounded from above. Taken together these facts yield that
$$\int_X\|\alpha_k-\alpha\|^2dV_\omega\to0.$$
Furthermore, since $\dbar\alpha_k=\dbar\alpha\ast\chi_k$ it follows that $\dbar\alpha_k$ converges to $\dbar\alpha$ as well.

Only the convergence of $\dbar^*_h\alpha_k$ remains. This is much more involved than the previous cases but follows from a classical result known as the Friedrich lemma that we will omit (\cite{F2},\cite{H}).
\end{proof}

With Corollary \ref{cor:main} and Lemma \ref{lma:appr} at our disposal the proof of Theorem \ref{thm:1.1} is now immediate.

\begin{proof}[Proof of Theorem 1.1]
For $\alpha\in Dom(T^*)\cap Ker(\dbar)$ we let $\{\alpha_k\}$ be the approximating sequence given by Lemma \ref{lma:appr}. Then by Corollary \ref{cor:main} for any fix $k$
$$p\delta\int_X\|\alpha_k\|^2dV_\omega\leq\int_X(\|\dbar_h^*\alpha_k\|^2+\|\dbar\alpha_k\|^2)dV_\omega.$$
Letting $k\to\infty$ and using Lemma \ref{lma:appr} once again together with Theorem \ref{thm:3.1} finishes the proof.
\end{proof} 

\begin{remark}\label{remark:1}
It is possible to extend Theorem \ref{thm:1.1} to any complex manifold $X$ carrying some complete K\"ahler metric. Hence, in particular, Theorem \ref{thm:1.1} holds for Stein manifolds (see e.g. \cite{D1}, Chapter VIII, Theorem 4.5). However as the argument is exactly the same as in the line bundle setting, (c.f. \cite{B} section 3.6), we will refrain from doing this here.
\end{remark}

\section{Proof of Theorem \ref{thm:1.3}}
\noindent The goal of this section is to prove Theorem \ref{thm:1.3}, and so throughout this section $X$ denotes a Riemann surface, $E$ is a holomorphic vector bundle on $X$, and $h$ is a singular hermitian metric on $E$ that is positively curved in the sense of Definition \ref{df:2.3}. Just as in the line bundle setting the proof will be carried out in several steps.

We start by assuming that $X$ is a non-compact Riemann surface and that $E$ is a trivial bundle. It is a consequence of the Behnke-Stein approximation theorem (which e.g. can be found in \cite{F1}, section 25), that every non-compact Riemann surface is a Stein manifold. Hence we can without any loss of generality assume that $X$ is a submanifold of $\C^N$, as any Stein manifold can be properly imbedded in $\C^N$, for some $N$. Let $D$ be a relatively compact Stein subdomain of $X$. Since the dimension of $X$ is lower than $N$, we cannot convolute $h$ with an approximate identity on $D$ immediately. However if we let $D'$ be a larger Stein open subset of $X$ containing $D$ in its interior, then there exists a neighborhood $U$ of $D'$ in $\C^N$, such that $U$ is a holomorphic retract of $U$, i.e. there exists a holomorphic map $p$ from $U$ to $D'$ which is the identity on $D'$, (see e.g. \cite{S2}, Corollary 1).

We claim that $p^*h=h\circ p$ is a singular hermitian metric on $U$ that is also positively curved in the sense of Definition \ref{df:2.3}. This is a straightforward consequence of two well-known facts. The first is that the composite of a plurisubharmonic and a holomorphic function is again plurisubharmonic, which immediately yields property (i), and the second is that $p^*(\Theta^h)=\Theta^{p^*h}$, from which property (ii) follows.

Through the frequently utilized regularisation technique of convolution with an approximative identity, we obtain a sequence of smooth hermitian metrics $\{h_\nu\}$ approximating $p^*h$ on $U$. One can now proceed to show that this sequence is increasing pointwise to $p^*h$, and that each metric $h_\nu$ is strictly positively curved. This is basically the content of Proposition 5.1 in \cite{R1} and so we will not repeat the argument here.

There exists however a slight difference between the two cases which stems from the fact that in \cite{R1} it is $h$, and not $p^*h$, that is approximated. Thus in the former setting, the regularising metrics are shown to have the same lower bound $\delta\omega\otimes I$, which we from now on will normalize so that $\delta=1$, on their curvature tensors. In our setting the same argument yields that the curvature of $h_\nu$ is bounded from below by $(p^*(\omega)\ast\chi_\nu)\otimes I$, (instead of $(\omega\ast\chi_\nu)\otimes I$), where $\chi_\nu$ denotes the approximate identity and $\omega$ is a K\"ahler form.

Now as $p^*(\omega)$ is smooth, $p^*(\omega)\ast\chi_\nu$ will converge to $p^*(\omega)$ in the $C^\infty$-topology. By definition, this means that for large enough $\nu$,
$$p^*(\omega)\ast\chi_\nu-p^*(\omega)\geq-\varepsilon\omega',$$
for some $\varepsilon>0$ and some positive metric $\omega'$ on $U$. When restricted to $X$, $\omega'|_X$ will still be positive, and so there exists a constant $c>0$ such that $\omega'|_X\leq c\omega$. Hence on $X$
$$p^*(\omega)\ast\chi_\nu|_X\geq p^*(\omega)|_X-\varepsilon\omega'|_X\geq(1-\varepsilon)\omega,$$
where we have normalized so that $c=1$.

All in all we get that on $D$, $h$ can be approximated by a sequence $\{h_\nu\}$ of smooth hermitian metrics increasing pointwise to $h$, and furthermore that each $h_\nu$ is strictly positively curved, all with the same constant $1-\varepsilon$.

For fixed $\nu$, we can now apply Theorem \ref{thm:1.1} in the complete K\"ahler case (see Remark \ref{remark:1}), and conclude that on $D$, for any $E$-valued, $(1,1)$-form $f$, we can solve $\dbar u=f$ with $u=u_\nu$ satisfying
$$\int_D\|u_\nu\|^2_{\omega,h_\nu}dV_\omega\leq\frac{1}{1-\varepsilon}\int_D\|f\|^2_{\omega,h_\nu}dV_\omega\leq\frac{1}{1-\varepsilon}\int_D\|f\|^2_{\omega,h}dV_\omega.$$
Just as in the line bundle case we can now argue that since $\{h_\nu\}$ are increasing, for $\nu>\nu_0$
$$\int_D\|u_\nu\|^2_{\omega,h_{\nu_0}}dV_\omega\leq\frac{1}{1-\varepsilon}\int_D\|f\|^2_{\omega,h}dV_\omega.$$
Hence we can deduce that there exists a subsequence of $\{u_\nu\}$ converging weakly in $L^2_{(1,0)}(D,E,h_{\nu_0})$. Using a diagonal argument it follows that there exists a subsequence of $\{u_\nu\}$ converging weakly in $L^2_{(1,0)}(D,E,h_{\nu_0})$ for any $\nu_0$. If we denote the weak limit by $u_D$ we hence have that $u_D$ solves $\dbar u=f$ in $D$, and furthermore 
$$\int_D\|u_D\|^2_{\omega,h_{\nu_0}}dV_\omega\leq\frac{1}{1-\varepsilon}\int_D\|f\|^2_{\omega,h}dV_\omega.$$
Letting $\nu_0$ tend to infinity and using monotone convergence we get that $u_D$ satisfies the required estimate in $D$.

Lastly, let $\{D_j\}$ denote an exhaustion of $X$ where each $D_j$ is a relatively compact, Stein subdomain. Then on each $D_j$ we can, by the above, solve $\dbar u=f$ with
$$\int_{D_j}\|u_{D_j}\|^2_{\omega,h}dV_\omega\leq\frac{1}{1-\varepsilon}\int_{D_j}\|f\|^2_{\omega,h}dV_\omega\leq\frac{1}{1-\varepsilon}\int_X\|f\|^2_{\omega,h}dV_\omega.$$
Hence, we once again have a uniform bound on the left hand side. Letting $\varepsilon$ tend to zero, repeating the argument above with respect to $j$, (instead of $\nu$) and taking weak limits again, finishes the proof of Theorem \ref{thm:1.3} when $X$ is a non-compact Riemann surface and $E$ is trivial.

The next step is to generalise this to the case when $X$ is still non-compact, but we no longer require $E$ to be trivial. However since we are assuming that $X$ is a Riemann surface, as long as $X$ is non-compact, $E$ will always be trivial, since every holomorphic vector bundle on a non-compact Riemann surface is holomorphically trivial (see e.g. \cite{F1} Theorem 30.4).

Finally we have the case when $X$ is a compact Riemann surface. For this we will utilize the following lemma.

\begin{lma}
Let $X$ be a Riemann surface and let $S$ be a point in $X$. Let $u$ and $f$ be (possibly bundle valued) forms in $L^2_{loc}(X)$ satisfying $\dbar u=f$ outside of $S$. Then the same equation holds on all of $X$ (in the sense of distributions).
\end{lma}

A more general form of this lemma, (where $X$ is any complex manifold and $S$ is a hypersurface), is used in the line bundle setting. The fact that $X$ is one dimensional and $S$ is a point simplifies the proof slightly, but since the more general result is well-known (see e.g. \cite{B} Lemma 5.1.3), we will refrain from proving it here.

Now if $X$ is a compact Riemann surface, by removing a point, $S$, we get a non-compact Riemann surface. By our previous result, we can solve $\dbar u=f$ on $X\setminus S$, and the weighted $L^2$ estimates for $u$ shows that it is locally in unweighted $L^2$ on $X\setminus S$ as well, (recall that the metric $h$ is assumed to be continuous). Thus by the previous lemma, $u$ actually solves the $\dbar$-equation across $S$ too.

\end{document}